\documentclass[a4paper, 12pt]{amsart}
\usepackage{graphicx}
\usepackage{amsmath,amssymb}

\newtheorem{thm}{Theorem}[section]
\newtheorem{cor}[thm]{Corollary}
\newtheorem{lem}[thm]{Lemma}
\newtheorem{prop}[thm]{Proposition}

\theoremstyle{definition}

\theoremstyle{remark}
\newtheorem{rem}[thm]{Remark}
\theoremstyle{example}

\theoremstyle{conjecture}

\numberwithin{equation}{section}

\def\dlim{\displaystyle\lim}
\def\dsum{\displaystyle\sum}

\newcommand{\C}{{\mathbb C}}

\newcommand{\N}{{\mathbb N}}

\newcommand{\re}{{\rm Re}\,}

\newcommand{\calC}{{\mathcal C}}

\newcommand{\calF}{{\mathcal F}}
\newcommand{\calI}{{\mathcal I}}
\newcommand{\calJ}{{\mathcal J}}

\newcommand{\calO}{{\mathcal O}}

\def\re{\mathop{\fam0 Re}}

\begin{document}

\title{Weighted composition operators between different Fock spaces}%

\author{Pham Trong Tien \& Le Hai Khoi$^\dag$}%
\address{(Tien) Department of Mathematics, Mechanics and Information Technology, Hanoi University of Science, VNU, 334 Nguyen Trai, Hanoi, Vietnam}%
\email{phamtien@mail.ru, phamtien@vnu.edu.vn}%%
\address{(Khoi) Division of Mathematical Sciences, School of Physical and Mathematical Sciences, Nanyang Technological University (NTU),
637371 Singapore}%
\email{lhkhoi@ntu.edu.sg}%%

\thanks{$^\dag$ Supported in part by MOE's AcRF Tier 1 grants M4011166.110 (RG24/13) and M4011724.110 (RG128/16)}



\date{\today}%

\begin{abstract}
We study weighted composition operators acting between Fock spaces. The following results are obtained:
\begin{itemize}
\item[(i)] Criteria for the boundedness and compactness.
\item[(ii)] Characterizations of compact differences and essential norm.
\item[(iii)] Complete descriptions of path connected components and isolated points of the space of composition operators and the space of nonzero weighted composition operators.
\end{itemize}
\end{abstract}

%----------------------------------------------------------------------------------------

\maketitle

%=====================
\section{Introduction}
%=====================

Let $X$ be a space of holomorphic functions on a domain $G$ in $\C$. For a holomorphic self-map $\varphi$ of $G$ and a holomorphic function $\psi$ on $G$, the \textit{weighted composition operator} $W_{\psi, \varphi}$ is defined by $W_{\psi, \varphi}(f): = \psi (f \circ \varphi)$ for $f \in X$. When the function $\psi$ is identically $1$, the operator $W_{\psi, \varphi}$ reduces to the \textit{composition operator} $C_{\varphi}$. 
A main problem in the investigation of such operators is to relate function theoretic properties of $\psi$ and $\varphi$ to operator theoretic properties of $C_{\varphi}$ and $W_{\psi, \varphi}$.

The study of composition operators on various Banach spaces of holomorphic functions on the unit disc or the unit ball, such as Hardy and Bergman spaces, the space $H^{\infty}$ of all bounded holomorphic functions, the disc algebra and weighted Banach spaces with sup-norm, etc. received a special attention of many authors during the past several decades (see \cite{CM, Sha} and references therein for more information).
Weighted composition operators on these spaces appeared in some works (see, for instance, \cite{C-78, C-80, CG-06,F-64}) with different applications. There is a great number of topics on operators of such a type: boundedness and compactness \cite{CH-03, CZ-04}, compact differences \cite{M-08}, topological structure \cite{BLW-09, HIO-05, IIO-12, IO-14}, dynamical and ergodic properties \cite{BCEJ-16, B-14, YR-07}. On many spaces, these topics are difficult and not yet solved completely.

Recently, much progress was made in the study of composition operators and weighted composition operators on Fock spaces. One of the main differences between operators $C_{\varphi}$ and $W_{\psi, \varphi}$ on Fock spaces and those on the above-mentioned spaces of holomorphic functions on the unit disc or the unit ball is the lack of bounded holomorphic functions in the Fock space setting. In fact, entire functions $\varphi$ that induce bounded composition operators $C_{\varphi}$ and  weighted composition operators $W_{\psi, \varphi}$ are quite restrictive, in details, they are only affine functions.
We refer the reader to \cite{CMS-03, D-14} for composition operators on the Hilbert Fock space $\calF^2(\C^n)$, to \cite{T-14, HK-16-1, U-07} for weighted composition operators on the Hilbert Fock space $\calF^2(\C)$. It should be noted that in these papers the techniques of adjoint operators in Hilbert spaces played an essential role.

The question to ask is: how about weighted composition operators acting between general Fock spaces $\calF^p(\C)$ and $\calF^q(\C)$ ($0 < p, q < \infty$). In this paper, we study several important questions for the operator $W_{\psi, \varphi}$: boundedness, compactness, essential norm, compact differences and topological structure. Roughly speaking, our main result is to give complete answers to all these questions by developing an essentially different approach without adjoint operators.

The paper is  organized as follows. Section 2 contains some preliminary results about the Fock spaces and operators defined on them. Section 3 deals with topological properties of weighted composition operators. In details, criteria for the boundedness and compactness of such operators are obtained. Note that in the case when $W_{\psi, \varphi}$ acts from a larger Fock space into a smaller one, these properties are equivalent. In view of this, we provide lower and upper estimates for essential norm of only weighted composition operators acting from $\calF^p(\C)$ into $\calF^q(\C)$ with $p \leq q$. In Section 4 we study the topological structure of the space of all composition operators and the space of all nonzero weighted composition operators between different Fock spaces endowed with the operator norm topology. We give complete characterizations of connected path components and isolated points in both these spaces. Necessary and sufficient conditions for the compactness of the difference of two weighted composition operators are also stated.

\section{Preliminaries}

For a number $p \in (0,\infty)$, the Fock space $\calF^p(\C)$ is defined as follows
$$
\calF^p(\C):=\left\{f\in\calO(\C): \|f\|_p = \left(\frac{p}{2\pi}\int_{\C}|f(z)|^p e^{-\frac{p|z|^2}{2}}\;dA(z)\right)^{1/p} < \infty\right\},
$$
where $\calO(\C)$ is the space of entire functions on $\C$ with the usual compact open topology and $dA$ is the Lebesgue measure on $\C$.
Furthermore, the space $\calF^{\infty}(\C)$ consists of all entire functions $f\in\calO(\C)$ for which
$$
\left\|f\right\|_{\infty}:=\sup_{z\in\C}|f(z)| e^{-\frac{\left|z\right|^2}{2}}< \infty.
$$

It is well known that $\calF^p(\C)$ with $p \geq 1$ and $\calF^{\infty}(\C)$ are Banach spaces. When $0 < p < 1$, $\calF^p(\C)$ is a complete metric space with the distance $d(f,g): = \|f-g\|^p_p$.

For each $w \in\C$, we define the function
$$
k_w(z)=e^{\overline{w}z-\frac{|w|^2}{2}},\ z\in\C.
$$
These functions play important roles in the study of Fock spaces $\mathcal F^p(\C)$. Obviously, $\|k_w\|_p = 1$ for every $w \in \C$ and $k_w$ converges to $0$ in $\mathcal O(\C)$ as $|w|\to\infty$.

We refer the reader to the monograph \cite{KZ} for more details about Fock spaces. Hereby, we give only some auxiliary results which will be needed in the sequel.

\begin{lem} \label{lem-Fp}
Let $p \in (0, \infty)$ be given. For each function $f \in \calF^p(\C)$, the following assertions are valid:
\begin{itemize}
\item[(i)]
$$
|f(z)| \leq e^{\frac{|z|^2}{2}}\|f\|_p,\ \forall z \in \C.
$$
\item[(ii)]
$$
|f'(z)| \leq e^2(1+|z|)e^{\frac{|z|^2}{2}} \|f\|_p,\ \forall z\in\C.
$$
\end{itemize}
\end{lem}

\begin{proof} (i) was proved in \cite[Corollary 2.8]{KZ}. 

(ii). Let $f\in\calF^p(\C)$. For $|z|\le1$, by the classical Cauchy formula and the part (i),
\begin{align*}
|f'(z)| & \leq \dfrac{1}{2 \pi} \int_{|\zeta - z| = 1} \dfrac{|f(\zeta)|}{|\zeta - z|^2} |d\zeta| \leq \max_{|\zeta - z| = 1} |f(\zeta)|\\
& \leq \|f\|_p \max_{|\zeta - z| = 1} e^{\frac{|\zeta|^2}{2}} \leq e^2 \|f\|_p.
\end{align*}

On the other hand, for $|z|>1$, arguing as above, we get
\begin{align*}
|f'(z)| & \leq \dfrac{1}{2 \pi} \int_{|\zeta - z| = |z|^{-1}} \dfrac{|f(\zeta)|}{|\zeta - z|^2} |d\zeta| \leq |z| \max_{|\zeta - z| = |z|^{-1}} |f(\zeta)| \\
& \le |z| e^{\frac{(|z| + |z|^{-1})^2}{2}} \|f\|_p  \le e^2 |z| e^{\frac{|z|^2}{2}} \|f\|_p.
\end{align*}

Combining these estimates yields the desired inequality.
\end{proof}

The following result was proved in \cite[Theorem 2.10]{KZ}.
\begin{lem}\label{lem-pq}
For $0< p < q < \infty$, $\mathcal F^p(\C) \subset \mathcal F^q(\C)$, and the inclusion is proper and continuous. Moreover,
$$
\|f\|_q \leq \left(\dfrac{q}{p}\right)^{\frac{1}{q}} \|f\|_p,\ \forall f \in \mathcal F^p(\C).
$$
\end{lem}

The following two lemmas give necessary and sufficient conditions for compactness of an operator acting from one Fock space into another.

\begin{lem}\label{lem-com}
Let $ p, q  \in (0, \infty)$ and $T$ be a linear continuous operator from $\mathcal O(\C)$ into itself and $T: \mathcal F^p(\C) \to \mathcal F^q(\C)$ be well-defined. The following two assertions are equivalent:
\begin{itemize}
\item[(i)] $T: \mathcal F^p(\C) \to \mathcal F^q(\C)$ is compact.
\item[(ii)] For every bounded sequence $(f_n)_n$ in $\mathcal F^p(\C)$ converging to $0$ in $\mathcal O(\C)$, the sequence $(Tf_n)_n$ also converges to $0$ in $\mathcal F^q(\C)$.
\end{itemize}
\end{lem}
\begin{proof}
(i) $\Rightarrow$ (ii). Suppose that $T: \mathcal F^p(\C) \to \mathcal F^q(\C)$ is compact and there is a bounded sequence  $(f_n)_n$ in $\mathcal F^p(\C)$ converging to $0$ in $\mathcal O(\C)$ such that $(Tf_n)_n$ does not converge to $0$ in $\mathcal F^q(\C)$.

Without loss of generality, we assume that there is a number $c>0$ such that
\begin{equation}\label{c-eq}
\|Tf_n\|_q \geq c, \ \forall n \in \N.
\end{equation}

Since $T: \mathcal F^p(\C) \to \mathcal F^q(\C)$ is compact, there is a subsequence $(f_{n_k})_k$ of $(f_n)_n$ such that $Tf_{n_k}$ converges to some function $g$ in $\mathcal F^q(\C)$.

On the other hand, since $T$ is continuous on $\mathcal O(\C)$, then $Tf_n$, and hence $Tf_{n_k}$ converge to $0$ in $\mathcal O(\C)$.

Consequently, the function $g$ must be identically zero which is a contradiction with \eqref{c-eq}. 

(ii) $\Rightarrow$ (i).
Let $B$ be an arbitrary bounded subset of $\mathcal F^p(\C)$ and $(f_n)_n$ be a sequence in $B$. By Lemma \ref{lem-Fp}[(i)] and Montel's theorem, $B$ is relatively compact in $\mathcal O(\C)$, and then there exists a subsequence $(f_{n_k})_k$ of $(f_n)_n$ converging to some function $f$ in $\mathcal O(\C)$. From this and Fatou's lemma, we have that $f \in \mathcal F^p(\C)$.

Therefore, the sequence $(f_{n_k}-f)_k$ is bounded in $\mathcal F^p(\C)$ and converges to $0$ in $\mathcal O(\C)$. By the hypothesis, $Tf_{n_k}$ also converges to $Tf$ in $\mathcal F^q(\C)$. 

Consequently, $TB$ is relatively compact in $\mathcal F^q(\C)$. 
\end{proof}

Note that the assumption that $T$ is a linear continuous operator on $\mathcal O(\C)$ plays an essential role in the proof of (i) $\Rightarrow$ (ii). Now, for an arbitrary operator $T: \mathcal F^p(\C) \to \mathcal F^q(\C)$ that would be not defined on $\mathcal O(\C)$, we get the following result.

\begin{lem}\label{lem-com1}
Let $p, q \in (1,\infty)$. If the operator $T: \mathcal F^p(\C) \to \mathcal F^q(\C)$ is compact, then for every sequence $(w_n)_n$ in $\C$ with $\dlim_{n \to \infty} |w_n| = \infty$, the sequence $(Tk_{w_n})_n$ converges to $0$ in $\mathcal F^q(\C)$.
\end{lem}

\begin{proof}
Since $p \in (1,\infty)$, for every sequence $(w_n)_n$ in $\C$ with $\dlim_{n \to \infty} |w_n| = \infty$, the sequence $(k_{w_n})_n$ weakly converges to $0$ in $\mathcal F^p(\C)$, and hence, $(Tk_{w_n})_n$ converges to $0$ in $\mathcal F^q(\C)$.
\end{proof}

For entire functions $\psi$ and $\varphi$ on $\C$, the following quantities play an important role in the present paper:
$$
m_z(\psi,\varphi):=|\psi(z)|e^{\frac{|\varphi(z)|^2-|z|^2}{2}},\ z\in\C,
$$
and
$$
m(\psi,\varphi):=\sup_{z\in\C}m_z(\psi,\varphi).
$$

%--------------------------------------------------------

\section{Topological properties}

\subsection{Boundedness and compactness}
In this subsection we study the boundedness and compactness for weighted composition operators acting from a Fock space $\mathcal F^p(\C)$ into an another one $\mathcal F^q(\C)$.

We obtain the following necessary condition.

\begin{prop}\label{prop-nec}
Let $p,q \in (0,\infty)$. If the weighted composition operator $W_{\psi, \varphi}: \mathcal F^p(\C) \to \mathcal F^q(\C)$ is bounded, then $\psi \in \mathcal F^q(\C)$ and $m(\psi, \varphi) < \infty$. In this case, $\varphi(z) = az+b$ with $|a| \leq 1$ and 
\begin{equation}\label{z-norm}
m_z(\psi,\varphi) \le \|W_{\psi,\varphi}k_{\varphi(z)}\|_q \le \left\|W_{\psi,\varphi}\right\|,\ \forall z\in\C.
\end{equation}
\end{prop}

\begin{proof}
Obviously, $\psi = W_{\psi, \varphi}(1) \in \mathcal F^q(\C)$.

For each $w \in \C$, using $\|k_w\|_p = 1$ and Lemma \ref{lem-Fp}[(i)], we have
\begin{align*}
\|W_{\psi,\varphi}\| & \geq \|W_{\psi, \varphi}k_w\|_q \geq |W_{\psi, \varphi}k_w(z)|e^{-\frac{|z|^2}{2}}\\
& = |\psi(z)| \big| e^{\overline{w}\varphi(z) - \frac{|w|^2}{2}} \big| e^{-\frac{|z|^2}{2}}, \ \forall z \in \C.
\end{align*}
In particular, with $w =\varphi(z)$, the last inequality means that
$$
m_z(\psi,\varphi) \leq \|W_{\psi, \varphi}k_{\varphi(z)}\|_q \leq \|W_{\psi,\varphi}\|, \ \forall z \in \C.
$$
Then $m(\psi,\varphi) \leq \|W_{\psi,\varphi}\|$. And hence, by \cite[Proposition 2.1]{T-14}, $\varphi(z) = az +b$ with $|a| \leq 1$.
\end{proof}

In view of Proposition \ref{prop-nec}, throughout this paper we always assume that $\psi$ is a nonzero function in $\calF^q(\C)$ and $\varphi(z) = az + b$ with $|a| \leq 1$. 

In the case $a=0$, from Proposition \ref{prop-nec} we get

\begin{cor}\label{cor-0}
Let $p,q \in (0,\infty)$ and $\psi$ be a nonzero function in $\calF^q(\C)$. If $a=0$, i.e., $\varphi(z) = b$, then the operator $W_{\psi, \varphi}: \mathcal F^p(\C) \to \mathcal F^q(\C)$ is compact and 
$$
\|W_{\psi, \varphi}\| \leq e^{\frac{|b|^2}{2}}\|\psi\|_q.
$$
\end{cor}
\begin{proof} 
By Lemma \ref{lem-Fp}[(i)], for each $f \in \mathcal F^p(\C)$,
$$
\|W_{\psi,\varphi}f\|_q = |f(b)|\|\psi\|_q \leq  e^{\frac{|b|^2}{2}} \|\psi\|_q \|f\|_p.
$$
Thus, the operator $W_{\psi, \varphi}: \mathcal F^p(\C) \to \mathcal F^q(\C)$ is bounded and 
$$
\|W_{\psi, \varphi}\| \leq e^{\frac{|b|^2}{2}}\|\psi\|_q.
$$
Moreover, $W_{\psi, \varphi}$ has rank $1$, and hence, it is compact.
\end{proof}

The case $0< |a| \leq 1$ is more complicated. At first, we consider weighted composition operators $W_{\psi, \varphi}$ acting from larger Fock spaces into smaller ones. In this case the boundedness and compactness of $W_{\psi, \varphi}$ are equivalent (see, Theorem \ref{thm-bd2} below).

To show this we will use the Berezin type integral transform
$$
B_{\psi,\varphi,q}(w): = \dfrac{q}{2 \pi}\int_{\C}|\psi(z)|^q \big|e^{\overline{w}\varphi(z) - \frac{|w|^2}{2}}\big|^q e^{-\frac{q|z|^2}{2}}dA(z) = \|W_{\psi, \varphi}k_w\|_q^q , \ w \in \C.
$$

Since $\varphi(z) = az + b$ with $0< |a| \leq 1$, we define the following positive pull-back measure $\mu_{\psi, \varphi, q}$ on $\C$ with 
$$
\mu_{\psi, \varphi, q}(E): = \dfrac{q}{2\pi}\int_{\varphi^{-1}(E)}|\psi(z)|^qe^{-\frac{q|z|^2}{2}}dA(z)
$$
for every Borel subset $E$ of $\C$.

We recall, for the reader's convenience, that for $p, q \in (0, \infty)$ a positive Borel measure $\mu$ on $\C$ is called a \textit{$(p,q)$-Fock Carleson measure} , if the embedding operator $i: \calF^p(\C) \to L^q(\C, d\mu)$ is bounded, i.e. there exists a constant $C>0$ such that for every $f \in \calF^p(\C)$,
$$
\left(\int_{\C} |f(z)|^q e^{-\frac{q|z|^2}{2}}d\mu(z) \right)^{\frac{1}{q}} \leq C \|f\|_p.
$$
We will write $\|\mu\|$ for the operator norm of $i$ from $\calF^p(\C)$ into $L^q(\C, d\mu)$ and refer the reader to \cite[Section~3]{HL-10} for more information about $(p,q)$-Fock Carleson measure.

\begin{thm} \label{thm-bd2}
Let $0< q < p < \infty$ and $\psi$ be a nonzero function in $\calF^q(\C)$ and $\varphi(z) = az + b$ with $0 < |a| \leq 1$. The following assertions are equivalent:
\begin{itemize}
 \item[(i)] The operator $W_{\psi, \varphi}: \mathcal F^p(\C) \to \mathcal F^q(\C)$ is bounded.
 \item[(ii)] The operator $W_{\psi, \varphi}: \mathcal F^p(\C) \to \mathcal F^q(\C)$ is compact.
 \item[(iii)] $m_z(\psi, \varphi) \in L^{\frac{pq}{p-q}}(\C,dA)$.
\end{itemize}
In this case, 
$$
\|W_{\psi, \varphi}\| \simeq \|m_z(\psi, \varphi)\|_{L^{\frac{pq}{p-q}}}.
$$
\end{thm}
\begin{proof}
(ii) $\Rightarrow$ (i) is obvious.

\medskip

(i) $\Rightarrow$ (iii). Assume that the operator $W_{\psi, \varphi}: \mathcal F^p(\C) \to \mathcal F^q(\C)$ is bounded. Then for each $f \in \mathcal F^p(\C)$,
\begin{align*}
\|W_{\psi, \varphi}\| \|f\|_p &\geq \|W_{\psi,\varphi}f\|_q = \left( \dfrac{q}{2 \pi} \int_{\C} |\psi(z)|^q |f(\varphi(z))|^q e^{-\frac{q|z|^2}{2}}dA(z) \right)^{\frac{1}{q}}\\
& = \left( \int_{\C} |f(z)|^q d\mu_{\psi, \varphi, q}(z) \right)^{\frac{1}{q}} = \left( \int_{\C}|f(z)|^q e^{-\frac{q|z|^2}{2}}d\lambda_{\psi, \varphi, q}(z) \right)^{\frac{1}{q}},
\end{align*}
where $d\lambda_{\psi, \varphi, q}(z) = e^{\frac{q|z|^2}{2}}d\mu_{\psi, \varphi, q}(z)$. The last inequality means that $\lambda_{\psi, \varphi, q}$ is a $(p,q)$ Fock-Carleson measure. Then by \cite[Theorem 3.3]{HL-10}, we get
$$
\widetilde{\lambda_{\psi, \varphi, q}}(w): = \int_{\C}|k_w(z)|^q e^{-\frac{q|z|^2}{2}}d\lambda_{\psi, \varphi, q}(z) \in L^{\frac{p}{p-q}}(\C,dA).
$$
Clearly, for all $w \in \C$,
\begin{align*}
\widetilde{\lambda_{\psi, \varphi, q}}(w)& = \int_{\C}|k_w(z)|^q e^{-\frac{q|z|^2}{2}}d\lambda_{\psi, \varphi, q}(z) = \int_{\C}|k_w(z)|^q d\mu_{\psi, \varphi, q}(z) \\
& = \dfrac{q}{2\pi} \int_{\C} |\psi(z)|^q |k_w(\varphi(z))|^q e^{-\frac{q|z|^2}{2}}dA(z) = B_{\psi, \varphi, q}(w).
\end{align*}
Consequently, $B_{\psi, \varphi, q}(w) \in L^{\frac{p}{p-q}}(\C,dA)$.

On the other hand, using Lemma \ref{lem-Fp}[(i)], we have that, for all $w, z \in \C$,
\begin{align*}
B_{\psi,\varphi,q}(w) & = \|W_{\psi, \varphi}k_w\|_q^q \geq |W_{\psi, \varphi}k_w(z)|^q e^{-\frac{q|z|^2}{2}}\\
& = |\psi(z)|^q \big| e^{\overline{w}\varphi(z) - \frac{|w|^2}{2}} \big|^q e^{-\frac{q|z|^2}{2}}.
\end{align*}
In particular, with $w = \varphi(z)$, we have
$$
B_{\psi,\varphi,q}(\varphi(z)) \geq m_z(\psi, \varphi)^{q}, \ \forall z \in \C,
$$
and hence
\begin{align}\label{MB}
\int_{\C} m_z(\psi, \varphi)^{\frac{pq}{p-q}}dA(z) &\leq \int_{\C} B_{\psi,\varphi,q}(\varphi(z))^{\frac{p}{p-q}}dA(z) \\
& = |a|^{-2}\int_{\C} B_{\psi,\varphi,q}(w)^{\frac{p}{p-q}}dA(w) < \infty \notag.
\end{align}
Thus, $m_z(\psi, \varphi) \in L^{\frac{pq}{p-q}}(\C,dA)$.

Moreover, by \cite[Theorem 3.3]{HL-10}, 
\begin{align*}
\|W_{\psi, \varphi}\|^q = \|\lambda_{\psi, \varphi, q}\|^q \simeq \|\widetilde{\lambda_{\psi, \varphi, q}}\|_{L^{\frac{p}{p-q}}} = \|B_{\psi, \varphi, q}\|_{L^{\frac{p}{p-q}}}.
\end{align*}
From this and \eqref{MB} it follows that
\begin{align}
\|m_z(\psi, \varphi)\|_{L^{\frac{pq}{p-q}}} & \leq  \left( |a|^{-2} \int_{\C} B_{\psi,\varphi,q}(z)^{\frac{p}{p-q}}dA(z) \right)^{\frac{p-q}{pq}} \nonumber \\
& \simeq |a|^{-\frac{2(p-q)}{pq}} \| W_{\psi, \varphi} \| \label{est-1}.
\end{align}

\medskip

(iii) $\Rightarrow$ (ii). 
For each function $f \in \mathcal F^p(\C)$, using H\"older's inequality, we obtain
\begin{align*}
\|W_{\psi, \varphi}f\|_q^q &= \dfrac{q}{2\pi} \int_{\C}  m_z(\psi, \varphi)^{q} |f(\varphi(z))|^q e^{-\frac{q|\varphi(z)|^2}{2}}dA(z)\\
& \leq \dfrac{q}{2\pi}\left(\int_{\C} |f(\varphi(z))|^p e^{-\frac{p|\varphi(z)|^2}{2}} dA(z)\right)^{\frac{q}{p}} \left(\int_{\C} m_z(\psi,\varphi)^{\frac{pq}{p-q}}dA(z)\right)^{\frac{p-q}{p}} \\
& \leq \dfrac{q}{2\pi} \left(\dfrac{2\pi}{p |a|^2}\right)^{\frac{q}{p}} \|f\|_p^q \left(\int_{\C} m_z(\psi,\varphi)^{\frac{pq}{p-q}}dA(z)\right)^{\frac{p-q}{p}} \\
& \leq \dfrac{q}{2\pi} \left(\dfrac{2\pi}{p|a|^2}\right)^{\frac{q}{p}} \|f\|_p^q  \left( \|m_z(\psi, \varphi)\|_{L^{\frac{pq}{p-q}}} \right)^q.
\end{align*}
The last inequality means that $W_{\psi, \varphi}: \calF^p(\C) \to \calF^q(\C)$ is bounded and
\begin{equation} \label{est-2}
\|W_{\psi, \varphi}\| \leq \left( \dfrac{q}{2\pi} \right)^{\frac{1}{q}} \left(\dfrac{2\pi}{p|a|^2}\right)^{\frac{1}{p}} \|m_z(\psi, \varphi)\|_{L^{\frac{pq}{p-q}}}.
\end{equation}

Next, let $(f_n)_n$ be an arbitrary bounded sequence in $\mathcal F^p(\C)$ converging to $0$ in $\mathcal O(\C)$. For each $n \in \N$ and $R>0$,
\begin{align*}
\|W_{\psi,\varphi}f_n\|_q^q & = \dfrac{q}{2\pi}\int_{\C} |\psi(z)|^q |f_n(\varphi(z))|^q e^{-\frac{q|z|^2}{2}}dA(z) \\ 
& = \dfrac{q}{2\pi} \left( \int_{|z| \leq R} + \int_{|z| > R} \right) |\psi(z)|^q |f_n(\varphi(z))|^q e^{-\frac{q|z|^2}{2}}dA(z)\\
& = \calI(n,R) + \calJ(n,R).
\end{align*}
Obviously,
\begin{align*}
\calI(n,R) &\leq  \dfrac{q}{2\pi} \max_{|z| \leq R} |f_n(\varphi(z))|^q \int_{|z| \leq R}|\psi(z)|^q e^{-\frac{q|z|^2}{2}}dA(z) 
& \leq  \|\psi\|_q^q \max_{|z| \leq R} |f_n(\varphi(z))|^q.
\end{align*}
For $\calJ(n,R)$, again using H\"older's inequality, we get
\begin{align*}
\mathcal J(n,R) &= \dfrac{q}{2\pi} \int_{|z| > R}  m_z(\psi, \varphi)^{q} |f_n(\varphi(z))|^q e^{-\frac{q|\varphi(z)|^2}{2}}dA(z)\\
& \leq \dfrac{q}{2\pi}\left(\int_{|z| > R} |f_n(\varphi(z))|^p e^{-\frac{p|\varphi(z)|^2}{2}} dA(z)\right)^{\frac{q}{p}} \left(\int_{|z| > R} m_z(\psi,\varphi)^{\frac{pq}{p-q}}dA(z)\right)^{\frac{p-q}{p}} \\
& \leq \dfrac{q}{2\pi} \left(\dfrac{2\pi}{p|a|^2}\right)^{\frac{q}{p}} \|f_n\|_p^q \left(\int_{|z| > R} m_z(\psi,\varphi)^{\frac{pq}{p-q}}dA(z)\right)^{\frac{p-q}{p}} \\
& \leq M^q \left(\int_{|z| > R} m_z(\psi,\varphi)^{\frac{pq}{p-q}}dA(z)\right)^{\frac{p-q}{p}},
\end{align*}
where 
$$
M^q: = \dfrac{q}{2\pi} \left(\dfrac{2\pi}{p}\right)^{\frac{q}{p}} \sup_n \|f_n\|_p^q < \infty.
$$
Consequently, for every $R>0$, letting $n \to \infty$, we obtain
\begin{align*}
\limsup_{n \to \infty} \|W_{\psi, \varphi}f_n\|_q^q &\leq  \limsup_{n \to \infty} \big( \mathcal I(n,R) + \mathcal J(n,R) \big) \\
& \leq M^q \left(\int_{|z| > R} m_z(\psi,\varphi)^{\frac{pq}{p-q}}dA(z)\right)^{\frac{p-q}{p}}.
\end{align*}
Since $ m_z(\psi,\varphi) \in L^{\frac{pq}{p-q}}(\C,dA)$, letting $R \to \infty$, we conclude that $W_{\psi,\varphi}f_n$ converges to $0$ in $\mathcal F^q(\C)$ as $n \to \infty$.

Consequently, by Lemma \ref{lem-com}, the operator $W_{\psi, \varphi}: \calF^p(\C) \to \calF^q(C)$ is compact.

Moreover, the desired estimates for $\|W_{\psi, \varphi}\|$ follow from \eqref{est-1} and \eqref{est-2}.
\end{proof}

For weighted composition operators $W_{\psi, \varphi}$ acting from smaller Fock spaces into larger ones, we get the following result.

\begin{thm}\label{thm-bd1}
Let $0< p \leq q < \infty$ and $\psi$ be a nonzero function in $\calF^q(\C)$ and $\varphi(z) = az + b$ with $0 < |a| \leq 1$.
\begin{itemize}
 \item[(a)] The operator $W_{\psi, \varphi}: \mathcal F^p(\C) \to \mathcal F^q(\C)$ is bounded if and only if
$m(\psi, \varphi) < \infty$. Moreover, 
\begin{equation*}
m(\psi,\varphi) \leq \left\|W_{\psi,\varphi}\right\|  \leq \left(\frac{q}{p\left|a\right|^2}\right)^{1/q} m(\psi,\varphi). 
\end{equation*}

\item[(b)] The operator $W_{\psi, \varphi}: \mathcal F^p(\C) \to \mathcal F^q(\C)$ is compact if and only if $\displaystyle \lim_{|z| \to \infty} m_z(\psi, \varphi) = 0$.
\end{itemize}
\end{thm}
\begin{proof}
For $p=q$, the results were proved in \cite{HK-16}. Hereby we sketch the proof in the case $p \leq q$ for the sake of the completeness.

\medskip

(a) The necessity follows from Proposition \ref{prop-nec}. 
Now assume that $m(\psi, \varphi) < \infty$. Then using Lemma \ref{lem-pq}, we have that for every $f \in \mathcal F^p(\C)$,
\begin{align*}
\|W_{\psi, \varphi}f\|_q &\leq m(\psi, \varphi) \left( \dfrac{q}{2\pi}\int_{\C} |f(\varphi(z))|^q e^{-\frac{q|\varphi(z)|^2}{2}}dA(z) \right)^{\frac{1}{q}} \\
& = m(\psi,\varphi) \dfrac{1}{|a|^{\frac{2}{q}}}  \|f\|_q \leq m(\psi,\varphi) \left(\dfrac{q}{p|a|^2}\right)^{\frac{1}{q}} \|f\|_p.
\end{align*}

Consequently, $W_{\psi, \varphi}: \mathcal F^p(\C) \to \mathcal F^q(\C)$ is bounded and 
$$
\left\|W_{\psi,\varphi}\right\|  \leq \left(\frac{q}{p\left|a\right|^2}\right)^{\frac{1}{q}} m(\psi,\varphi),
$$
which and \eqref{z-norm} imply the desired estimates for $\|W_{\psi, \varphi}\|$.

\medskip

(b)
\textbf{Necessary.} Suppose that $W_{\psi,\varphi}: \mathcal F^p(\C) \to \mathcal F^q(\C)$ is compact.
For every sequence $(z_n)_n$ in $\C$ converging to $\infty$, we have that $k_{\varphi(z_n)}$ converges to $0$ in $\mathcal O(\C)$. Therefore, by \eqref{z-norm} and Lemma \ref{lem-com},
$$
m_{z_n}(\psi,\varphi) \leq \|W_{\psi,\varphi}k_{\varphi(z_n)}\|_q \to 0 \text{ as } n \to \infty.
$$
From this, $\lim_{|z| \to \infty} m_z(\psi,\varphi) = 0$.

\textbf{Sufficiency.} By part (a), the operator $W_{\psi, \varphi}: \calF^p(\C) \to \calF^q(\C)$ is bounded.

Let $(f_n)_n$ be an arbitrary bounded sequence in $\mathcal F^p(\C)$ converging to $0$ in $\mathcal O(\C)$.
Then for each $n \in \N$ and $R>0$, using Lemma \ref{lem-pq}, we have
\begin{align*}
\|W_{\psi,\varphi}f_n\|_q^q & = \dfrac{q}{2\pi}\int_{\C} |\psi(z)|^q |f_n(\varphi(z))|^q e^{-\frac{q|z|^2}{2}}dA(z) \\ 
& = \dfrac{q}{2\pi} \left( \int_{|z| \leq R} + \int_{|z| > R} \right) |\psi(z)|^q |f_n(\varphi(z))|^q e^{-\frac{q|z|^2}{2}}dA(z) \\
& \leq  \dfrac{q}{2\pi} \max_{|z| \leq R} |f_n(\varphi(z))|^q \int_{|z| \leq R}|\psi(z)|^q e^{-\frac{q|z|^2}{2}}dA(z) \\
& + \dfrac{q}{2\pi} \sup_{|z|>R} m_z(\psi,\varphi)^{q} \int_{|z|>R} |f_n(\varphi(z))|^q e^{-\frac{q|\varphi(z)|^2}{2}}dA(z) \\
& \leq  \|\psi\|_q^q \max_{|z| \leq R} |f_n(\varphi(z))|^q + \dfrac{\|f_n\|_q^q}{ |a|^2} \sup_{|z|>R} m_z(\psi,\varphi)^{q}\\
& \leq  \|\psi\|_q^q \max_{|z| \leq R} |f_n(\varphi(z))|^q + \dfrac{qM^q}{ p|a|^2} \sup_{|z|>R} m_z(\psi,\varphi)^{q},
\end{align*}
where $M: = \sup_n \|f_n\|_p < \infty$.

From this, letting $n \to \infty$, and then $R \to \infty$, we get that the sequence $W_{\psi,\varphi}f_n$ converges to $0$ in $\mathcal F^q(\C)$.

Therefore, by Lemma \ref{lem-com}, $W_{\psi, \varphi}$ is a compact operator from $\mathcal F^p(\C)$ into $\mathcal F^q(\C)$.
\end{proof}

To end this subsection we give a complete characterization for the boundedness and compactness of composition operators $C_{\varphi}$.

\begin{cor}\label{cor-co-1}
Let $0< p \leq q < \infty$.
\begin{itemize}
 \item[(a)] The operator $C_{\varphi}: \mathcal F^p(\C) \to \mathcal F^q(\C)$ is bounded if and only if 
\begin{equation}\label{eq-vp}
\varphi(z) = \begin{cases}
az + b, \ \ \text{ if } |a| < 1, \\
az, \ \ \ \ \ \ \text{ if } |a| = 1.  
\end{cases}
\end{equation}
\item[(b)] The operator $C_{\varphi}: \mathcal F^p(\C) \to \mathcal F^q(\C)$ is compact if and only if $\varphi(z) = az + b$ with $|a| < 1$.
\end{itemize}
\end{cor}
\begin{proof}
Clearly, $m(1, \varphi) < \infty$ if and only if $\varphi$ as in \eqref{eq-vp}.
Then the assertion immediately follows from Theorem \ref{thm-bd1} and Corollary \ref{cor-0}.
\end{proof}

\begin{cor}\label{cor-co-2}
Let $0< q < p < \infty$. The following assertions are equivalent:
\begin{itemize}
 \item[(i)] The operator $C_{\varphi}: \mathcal F^p(\C) \to \mathcal F^q(\C)$ is bounded.
\item[(ii)] The operator $C_{\varphi}: \mathcal F^p(\C) \to \mathcal F^q(\C)$ is compact.
\item[(iii)] $\varphi(z) = az + b$ with $|a| < 1$.
\end{itemize}
\end{cor}
\begin{proof}
We can easily show that for each affine function $\varphi(z) = az + b$ as in \eqref{eq-vp}, $m(1,\varphi) \in L^{\frac{pq}{p-q}}(\C,dA)$ if and only if $|a| < 1$.
Then the assertion follows from Theorem \ref{thm-bd2} and Corollary \ref{cor-0}.
\end{proof}
%----------------------------------------------------------

\subsection{Essential norm} 
In a general setting, let $X, Y$ be Banach spaces, and $\mathcal{K}(X,Y)$ be the set of all compact operators from $X$ into $Y$. The essential norm of a bounded linear operator $L: X \to Y$, denoted as $\|L\|_e$, is defined as
$$
\|L\|_e=\inf\{\|L-K\|: K\in\mathcal{K}(X,Y)\}.
$$
Clearly, $L$ is compact if and only if $\|L\|_e=0$.

In view of Corollary \ref{cor-0}, Theorem \ref{thm-bd2} and Lemma \ref{lem-com1}, we study essential norm of $W_{\psi, \varphi}: \mathcal F^p(\C) \to \mathcal F^q(\C)$ when $1< p \leq q < \infty$ and $\varphi(z)=az+b$ with $0< |a|\le1$.

The main result is stated as follows.

\begin{thm}\label{thm-essnorm}
Let $1 < p \leq q < \infty$ and $W_{\psi,\varphi}: \calF^p(\C) \to \calF^q(\C)$ be a bounded weighted composition operator induced by a nonzero entire function $\psi \in \calF^q(\C)$ and an affine function $\varphi(z)=az+b$ with $0 < |a| \leq 1$. Then
$$
\|W_{\psi,\varphi}\|_e \simeq \limsup_{|z|\to\infty} m_z(\psi,\varphi).
$$
More precisely,
$$
\limsup_{|z|\to\infty}m_z(\psi,\varphi) \leq \|W_{\psi,\varphi}\|_e \leq 2 \left(\dfrac{q}{p|a|^2}\right)^{\frac{1}{q}} \limsup_{|z|\to\infty}m_z(\psi,\varphi).
$$
\end{thm}

\begin{proof}
It is clear, by \eqref{z-norm}, that $\displaystyle\limsup_{|z|\to\infty}m_z(\psi,\varphi)$ is finite.

\noindent\textbf{- Lower estimate}. We prove the lower estimate for $\|W_{\psi, \varphi}\|_e$ by contradiction.
Assume in contrary that
$$
\|W_{\psi,\varphi}\|_e < \limsup_{|z|\to\infty} m_z(\psi,\varphi).
$$
Then there are positive constants $A<B$ and a compact operator $T$ acting from $\mathcal F^p(\mathbb C)$ into $\mathcal F^q(\mathbb C)$ such that
$$
\|W_{\psi,\varphi} - T\| < A < B < \limsup_{|z|\to\infty} m_z(\psi,\varphi).
$$
We can find a sequence $(z_n)$ with $|z_n|\uparrow\infty$ so that
\begin{equation} \label{eq-1}
\lim_{n \to \infty} m_{z_n}(\psi,\varphi) = \limsup_{|z|\to\infty} m_z(\psi,\varphi)> B.
\end{equation}
On the other hand, using \eqref{z-norm}, for each $n\in N$, we have
\begin{align*}
\|W_{\psi,\varphi} - T\| &\geq \|W_{\psi,\varphi}k_{\varphi(z_n)} - T k_{\varphi(z_n)}\|_q \geq \|W_{\psi,\varphi}k_{\varphi(z_n)}\|_q - \|T k_{\varphi(z_n)}\|_q \\
&\geq m_{z_n}(\psi,\varphi) - \|T k_{\varphi(z_n)}\|_q.
\end{align*}
Since $a\neq 0$, $\varphi(z_n)\to\infty$ as $n \to \infty$,  by Lemma \ref{lem-com1}, $\|T k_{\varphi(z_n)}\|_q \to 0$ as $n \to \infty$.

From this and \eqref{eq-1}, we obtain
$$
A > \|W_{\psi,\varphi} - T\| \geq \lim_{n\to\infty}m_{z_n}(\psi,\varphi) > B,
$$
which is a contradiction.

\medskip

\noindent\textbf{- Upper estimate}. For each $k\in\N$, we consider the dilation operator $U_{r_k}f(z): = f(r_kz)$ with $r_k: =k/(k+1)$. Then by Corollary \ref{cor-co-1} and Theorem \ref{thm-bd1}, $U_{r_k}$ is a compact operator from $\calF^p(\C)$ into $\calF^q(\C)$ and 
$$
\|U_{r_k}\| \leq \left(\dfrac{q}{p r_k^2}\right)^{\frac{1}{q}},\ \forall k\in\N.
$$
Take and fix a number $R > 0$. For every $k\in\N$, we have
\begin{align*}
\|W_{\psi, \varphi}\|_e &\leq \|W_{\psi, \varphi} - W_{\psi, \varphi} \circ U_{r_k}\| = \sup_{\|f\|_p \leq 1} \|W_{\psi, \varphi} \circ (I - U_{r_k})f\|_q \\
& = \sup_{\|f\|_p \leq 1} \left( \frac{q}{2\pi}\int_\C |W_{\psi, \varphi} \circ (I - U_{r_k})f(z)|^q e^{-\frac{q|z|^2}{2}}dA(z)\right)^{\frac{1}{q}} \\
& \leq \sup_{\|f\|_p \leq 1} \left( \frac{q}{2\pi}\int_{|z| \leq R} |W_{\psi, \varphi} \circ (I - U_{r_k})f(z)|^q e^{-\frac{q|z|^2}{2}}dA(z)\right)^{\frac{1}{q}} \\
& + \sup_{\|f\|_p \leq 1} \left( \frac{q}{2\pi}\int_{|z| > R} |W_{\psi, \varphi} \circ (I - U_{r_k})f(z)|^q e^{-\frac{q|z|^2}{2}}dA(z)\right)^{\frac{1}{q}} \\
&=\calI(R,k) + \calJ(R,k),
\end{align*}
where $I$ is the inclusion from $\mathcal F^p(\C)$ into $\mathcal F^q(\C)$ with $\|I\| \leq \left( \frac{q}{p} \right)^{\frac{1}{q}}$.

On one hand, we observe that
\begin{align*}
\calJ(R,k) & \leq \sup_{|z|>R}m_z(\psi,\varphi) \sup_{\|f\|_p \leq 1}\left( \dfrac{q}{2\pi}\int_{|z| > R} |(I - U_{r_k})f(\varphi(z))|^q e^{-\frac{q|\varphi(z)|^2}{2}}dA(z)\right)^{\frac{1}{q}}\\
& \leq |a|^{-\frac{2}{q}} \sup_{|z|>R}m_z(\psi,\varphi) \sup_{\|f\|_p \leq 1} \|(I - U_{r_k})f\|_q \\
& = |a|^{-\frac{2}{q}} \|I - U_{r_k}\| \sup_{|z|>R}m_z(\psi,\varphi)  \\
&\leq |a|^{-\frac{2}{q}} \left(\left(\dfrac{q}{p}\right)^{\frac{1}{q}}+\left(\dfrac{q}{pr_k^2}\right)^{\frac{1}{q}}\right) \sup_{|z|>R}m_z(\psi,\varphi).
\end{align*}
On the other hand, we have
\begin{align*}
\calI(R,k) & \leq \sup_{|z| \leq R}m_z(\psi,\varphi) \sup_{\|f\|_p \leq 1}\left( \dfrac{q}{2\pi}\int_{|z| \leq R} |(I - U_{r_k})f(\varphi(z))|^q e^{-\frac{q|\varphi(z)|^2}{2}}dA(z)\right)^{\frac{1}{q}}\\
& \leq \left( \dfrac{q}{2\pi}\int_{|z| \leq R} e^{-\frac{q|\varphi(z)|^2}{2}}dA(z)\right)^{\frac{1}{q}} \sup_{|z| \leq R}m_z(\psi,\varphi) \sup_{\|f\|_p \leq 1} \sup_{|z| \leq R} \big|(I - U_{r_k})f(\varphi(z))\big| \\
& \leq |a|^{-\frac{2}{q}}\|1\|_q m(\psi,\varphi) \sup_{\|f\|_{\infty} \leq 1} \sup_{|z| \leq R}\big|(I - U_{r_k})f(\varphi(z))\big|,
\end{align*}
where the last inequality is based on the fact that $\|f\|_{\infty} \leq \|f\|_p$ for every $f\in\calF^p(\C)$.

For each $f(z)=\dsum_{j=0}^{\infty} a_j z^j $ with $\|f\|_{\infty} \leq 1$, we have
$$
|a_j| = \dfrac{|f^{(j)}(0)|}{j!} \leq \dfrac{1}{\rho^j} \max_{|\zeta| = \rho} |f(\zeta)| \leq \frac{e^{\frac{\rho^2}{2}}}{\rho^j},\ \forall j\geq1, \rho >0,
$$
which gives
$$
|a_j| \leq \inf_{\rho > 0} \dfrac{e^{\frac{\rho^2}{2}}}{\rho^j} = \left(\frac{e}{j}\right)^{\frac{j}{2}},\ \text{for all}\ j \geq1.
$$
Putting $ \displaystyle R_{\varphi}: = \max_{|z| \leq R} |\varphi(z)|$, we obtain
\begin{align*}
\calI(R,k) & \leq |a|^{-\frac{2}{q}}\|1\|_q m(\psi,\varphi) \sup_{\|f\|_{\infty} \leq 1} \sup_{|z| \leq R} \big|(I - U_{r_k})f(\varphi(z))\big| \\
& \leq |a|^{-\frac{2}{q}}\|1\|_q m(\psi,\varphi) \sup_{\|f\|_{\infty} \leq 1} \sup_{|z| \leq R_{\varphi}}\big|(I - U_{r_k})f(z)\big| \\
& \leq |a|^{-\frac{2}{q}}\|1\|_q m(\psi,\varphi) \sup_{\|f\|_{\infty} \leq 1} \sup_{|z| \leq R_{\varphi}} \sum_{j=1}^{\infty} |a_j| \left(1 - \frac{k^j}{(k+1)^j} \right) |z|^j \\
& \leq \dfrac{|a|^{-\frac{2}{q}}\|1\|_q}{k+1} m(\psi,\varphi) \sum_{j=1}^{\infty} j R_{\varphi}^j \left(\frac{e}{j}\right)^{\frac{j}{2}}.
\end{align*}
Consequently,
\begin{align*}
\|W_{\psi, \varphi}\|_e &\leq \limsup_{k \to \infty} \|W_{\psi, \varphi} - W_{\psi, \varphi} \circ U_{r_k}\| \leq \limsup_{k \to \infty}\calI(R,k) + \limsup_{k \to \infty}\calJ(R,k) \\
& \leq 2 \left(\dfrac{q}{p|a|^2}\right)^{\frac{1}{q}}\sup_{|z|>R}m_z(\psi,\varphi),
\end{align*}
from which the upper estimate of $\|W_{\psi,\varphi}\|_e$ follows by letting $R \to \infty$.
\end{proof}

From this result, we have the following simplified estimates for the essential norm of a bounded weighted composition operator $W_{\psi,\varphi}$ in case $|a|=1$.

\begin{cor}\label{ess-noncpt}
If $|a|=1$, then
$$
|\psi(0)|e^{\frac{|b|^2}{2}} \leq \|W_{\psi,\varphi}\|_e \leq 2\left(\dfrac{q}{p}\right)^{1/q} |\psi(0)|e^{\frac{|b|^2}{2}}.
$$
\end{cor}

\begin{rem}
Ueki \cite{U-07} showed that the essential norm of $W_{\psi, \varphi}$ on Hilbert space $\mathcal F^2(\C)$ is equivalent to $\displaystyle \limsup_{|z|\to \infty} B_{\varphi}(|\psi|^2)(z)$, where $B_{\varphi}(|\psi|^2)(z)$ is the integral transform
$$
B_{\varphi}(|\psi|^2)(z) = \int_{\C} |\psi(\zeta)|^2 \left|e^{\left\langle\varphi(\zeta),z\right\rangle}\right|^2 e^{-|\zeta|^2} e^{-|z|^2}dA(\zeta).
$$
However, this result is difficult to use, even for composition operators, that is, when $\psi$ is a constant function.

Our Theorem \ref{thm-essnorm} is simpler and more effective for essential norm of $W_{\psi, \varphi}$ acting from smaller general Fock spaces $\mathcal F^p(\C)$ into larger ones $\mathcal F^q(\C)$.
Moreover, Theorem \ref{thm-essnorm} also give an answer response to T.~Le's question in \cite[Remark 2.5]{T-14}.
\end{rem}

%---------------------------------------------------------

\section{Topological structure}
%===========================================================================
One of the recent main subjects in the study of (weighted) composition operators is related to the topological structure of the space of such operators endowed with the operator norm topology. 

In a general setting, let $X$ and $Y$ be two spaces of holomorphic functions on a domain $G$. For every bounded weighted composition operator $W_{\psi, \varphi}: X \to Y$, we can easily show that $W_{\psi, \varphi}$ and the zero operator $0$ belong to the same path connected component in the space of weighted composition operators acting from $X$ into $Y$ via the path $T_t: = W_{t\psi, \varphi}$ for $t \in [0,1]$. Then researchers study the topological structure for the space of only nonzero weighted composition operators from $X$ into $Y$. We write $\calC(X,Y)$ for the space of composition operators and $\calC_w(X,Y)$ for the space of nonzero weighted composition operators acting from $X$ into $Y$ under the operator norm topology.
Acording to \cite{SS-90}, the important problems in this topic were raised as follows:
\begin{itemize}
\item[(i)] Characterize the components of $\calC(X,Y)$ and $\calC_w(X,Y)$.
\item[(ii)] Characterize isolated points $\calC(X,Y)$ and $\calC_w(X,Y)$.
\item[(iii)] Characterize compact differences of (weighted) composition operators. 
\end{itemize} 
These questions have been intensively investigated on Bergman spaces \cite{M-05}, on Hardy spaces \cite{GGNS-08, IO-14}, on the space $H^{\infty}$ of bounded holomorphic functions \cite{IIO-12, MOZ-01}, on weighted Banach spaces of holomorphic functions with sup-norm \cite{BLW-09, M-08}, on Hilbert Fock space $\calF^2(\C^n)$ \cite{D-14}.

In this section we investigate the topological structure for both spaces $\calC \big(\calF^p(\C),\calF^q(\C)\big)$ and $\calC_w\big(\calF^p(\C),\calF^q(\C)\big)$ with $p, q \in (0, \infty)$ and give complete answers to all the mentioned-above questions.

\subsection{Compact differences}

In view of Theorem \ref{thm-bd2} we will study the compactness of the difference of two bounded weighted composition operators acting from a smaller Fock space $\calF^p(\C)$ into another larger one $\calF^q(\C)$.

\begin{thm}\label{thm-comdif-wco}
Let $0< p \leq q < \infty$ and $W_{\psi_1,\varphi_1}$ and $W_{\psi_2,\varphi_2}$ be two weighted composition operators in $\calC_w\big(\calF^p(\C),\calF^q(\C)\big)$ induced respectively by nonzero entire functions $\psi_1, \psi_2 \in \calF^q(\C)$ and affine functions $\varphi_1(z)=a_1z+b_1, \varphi_2(z)=a_2z+b_2$ with $|a_1|\le1, |a_2|\le1$. Then the difference $W_{\psi_1, \varphi_1} - W_{\psi_2, \varphi_2}: \calF^p(\C) \to \calF^q(\C)$ is compact if and only if either of the following conditions is satisfied:
\begin{itemize}
\item[(i)] Both $W_{\psi_1,\varphi_1}$ and $W_{\psi_2,\varphi_2}$ are compact operators from $\calF^p(\C)$ into $\calF^q(\C)$.
\item[(ii)] $\varphi_1 = \varphi_2 := \varphi$ and  $\dlim_{|z| \to \infty} m_z(\psi_1-\psi_2,\varphi)=0.$
\end{itemize}
\end{thm}

\begin{proof} Since $W_{\psi_1,\varphi}-W_{\psi_2,\varphi}=W_{\psi_1-\psi_2,\varphi}$, the sufficiency follows from Theorem \ref{thm-bd1}.

For the necessity, suppose that the difference $W_{\psi_1,\varphi_1} - W_{\psi_2,\varphi_2}: \calF^p(\C) \to \calF^q(\C)$ is compact. Then both $W_{\psi_1,\varphi_1}$ and $W_{\psi_2,\varphi_2}$ must be either compact or non-compact on $\calF^p(\C)$ simultaneously.

Consider the case when both $W_{\psi_1,\varphi_1}$ and $W_{\psi_2,\varphi_2}$ are non-compact. From Theorem \ref{thm-bd1} it follows that
$$
c_k=\limsup_{|z|\to\infty} m_z(\psi_k,\varphi_k)>0\ (k=1,2).
$$
Then for say $c_1$, there exists a sequence $(z_n)$ with $|z_n| \uparrow \infty$ as $n \to \infty$, such that
$$
\lim_{n\to\infty} m_{z_n}(\psi_1,\varphi_1) = \limsup_{|z|\to\infty} m_z(\psi_1,\varphi_1)=c_1.
$$
By Lemma \ref{lem-Fp}[(i)], for all $w,z\in\C$,
\begin{eqnarray*}
\|W_{\psi_1, \varphi_1}k_{w} - W_{\psi_2, \varphi_2}k_{w}\|_q%
&\ge& \left|W_{\psi_1, \varphi_1}k_{w}(z)- W_{\psi_2, \varphi_2}k_{w}(z)\right| e^{-\frac{|z|^2}{2}} \\
&=& \left|\psi_1(z)e^{\overline{w}\varphi_1(z)-\frac{|w|^2+|z|^2}{2}}-\psi_2(z) e^{\overline{w}\varphi_2(z)-\frac{|w|^2+|z|^2}{2}}\right|.
\end{eqnarray*}
In particular, with $w=\varphi_1(z)$, the last inequality gives
\begin{align}\label{ine}
& \|W_{\psi_1, \varphi_1}k_{\varphi_1(z)} - W_{\psi_2, \varphi_2}k_{\varphi_1(z)}\|_q \\
\ge& \left|\big|\psi_1(z)e^{\frac{|\varphi_1(z)|^2-|z|^2}{2}}\big| -\big|\psi_2(z) e^{\overline{\varphi_1(z)}\varphi_2(z)-\frac{|\varphi_1(z)|^2+|z|^2}{2}}\big|\right| \notag \\
=& \left| m_z(\psi_1, \varphi_1) - m_z(\psi_2, \varphi_2) e^{-\frac{ |\varphi_1(z)|^2+|\varphi_2(z)|^2 - 2\re \left( \overline{\varphi_1(z)}\varphi_2(z)\right)}{2}} \right| \notag \\
=& \left| m_z(\psi_1, \varphi_1) - m_z(\psi_2, \varphi_2) e^{-\frac{|\varphi_1(z) - \varphi_2(z)|^2}{2}} \right|
, \forall z \in \mathbb C. \notag
\end{align}

There are two cases for complex numbers $a_1$ and $a_2$.

\medskip
\noindent\textbf{- Case 1: $a_1\ne a_2$}. In this case, 
$$
\lim_{n\to\infty} e^{-\frac{|\varphi_1(z_n) - \varphi_2(z_n)|^2}{2}}=0.
$$
From this, taking into account the inequality $m_{z_n}(\psi_2,\varphi_2) < m(\psi_2,\varphi_2)$, for every $n \in \N$, it follows that
$$
\lim_{n \to \infty} m_{z_n}(\psi_2, \varphi_2)e^{-\frac{|\varphi_1(z_n) - \varphi_2(z_n)|^2}{2}} =0.
$$
Obviously, $W_{\psi_1,\varphi_1} - W_{\psi_2,\varphi_2}$ is a linear continuous operator on $\mathcal O(\C)$. Then, by Lemma \ref{lem-com},
$$
\lim_{n \to \infty} \|W_{\psi_1,\varphi_1}k_{\varphi_1(z_n)} - W_{\psi_2,\varphi_2}k_{\varphi_1(z_n)}\|_q = 0.
$$
Consequently, by \eqref{ine}, we get
\begin{align*}
c_1 & = \lim_{n\to\infty} \left( m_{z_n}(\psi_1,\varphi_1) - m_{z_n}(\psi_2,\varphi_2) e^{-\frac{|\varphi_1(z_n) - \varphi_2(z_n)|^2}{2}}\right)\\
&\le \lim_{n\to\infty} \left| m_{z_n}(\psi_1,\varphi_1) - m_{z_n}(\psi_2,\varphi_2) e^{-\frac{|\varphi_1(z_n) - \varphi_2(z_n)|^2}{2}}\right| \\
& \le \lim_{n \to \infty} \|W_{\psi_1,\varphi_1}k_{\varphi_1(z_n)} - W_{\psi_2,\varphi_2}k_{\varphi_1(z_n)}\|_q = 0,
\end{align*}
which is impossible.

\medskip
\noindent\textbf{- Case 2: $a_1=a_2=a$}. In this case, \eqref{ine} gives
\begin{align*}
& \|W_{\psi_1, \varphi_1}k_{\varphi_1(z)} - W_{\psi_2, \varphi_2}k_{\varphi_1(z)}\|_q \\
\geq & \left| m_z(\psi_1, \varphi_1) - m_z(\psi_2, \varphi_2) e^{-\frac{|b_1 - b_2|^2}{2}} \right|,\ \forall z\in\C.
\end{align*}
Moreover,
$$
\liminf_{n \to \infty} m_{z_n}(\psi_2,\varphi_2) \leq \limsup_{n \to \infty} m_{z_n}(\psi_2,\varphi_2) \leq \limsup_{|z| \to \infty} m_{z}(\psi_2,\varphi_2) = c_2.
$$
Therefore,
\begin{align*}
& c_1 - c_2 e^{-\frac{|b_1 - b_2|^2}{2}} \\
\le & \limsup_{n\to\infty} \left( m_{z_n}(\psi_1,\varphi_1) - m_{z_n}(\psi_2,\varphi_2) e^{-\frac{|b_1 - b_2|^2}{2}} \right) \\
\le & \lim_{n\to\infty} \|W_{\psi_1,\varphi_1}k_{\varphi_1(z_n)} - W_{\psi_2,\varphi_2}k_{\varphi_1(z_n)}\|_q = 0.
\end{align*}
Hence,
$$
c_1 \leq c_2 e^{-\frac{|b_1 - b_2|^2}{2}}.
$$

Interchanging the role of $\varphi_1$ and $\varphi_2$ in the proofs above, we also obtain
$$
c_2 \leq c_1 e^{-\frac{|b_1 - b_2|^2}{2}}.
$$

Combining the last two inequalities yields
$$
c_2 \le c_2 e^{-|b_1 - b_2|^2} \le c_2,
$$
which gives $b_1=b_2$.

Thus $\varphi_1(z)=\varphi_2(z) = \varphi(z) = az+b$, which gives $W_{\psi_1,\varphi_1} - W_{\psi_2,\varphi_2} = W_{\psi_1-\psi_2,\varphi}$. By Theorem \ref{thm-bd1}, $\dlim_{|z|\to\infty} m_z(\psi_1-\psi_2,\varphi)=0$.
\end{proof}
From this theorem we immediately get the following result for compact differences of two composition operators.
\begin{cor}\label{cor-comdif-co}
Let $0< p \leq q < \infty$. Then the difference of two distinct composition operators acting from $\calF^p(\C)$ into $\calF^q(\C)$ is compact if and only if both composition operators are compact.
\end{cor}

\subsection{The space $\mathcal C \big( \mathcal F^p(\C), \mathcal F^q(\C) \big)$} \label{topo-c}
In this subsection we give a complete description of path connected and connected components and isolated points of the space $\calC \big(\calF^p(\C),\calF^q(\C)\big)$.

\begin{prop} \label{prop-path-c}
Let $p, q \in (0, \infty)$ and $C_{\varphi}$ be a compact composition operator from $\calF^p(\C)$ into $\calF^q(\C)$ induced by entire function $\varphi(z)=az+b$ with $|a| < 1$. Then $C_{\varphi}$ and $C_{\varphi(0)}$ belong to the same path connected component of $\calC \big(\calF^p(\C),\calF^q(\C)\big)$.
\end{prop}

\begin{proof}
If $a = 0$ then the assertion is trivial. So we assume that $0 < |a| <1$.

For each $s \in [0,1]$, put $\varphi_s(z): = \varphi(sz), \ z \in \C$. Then, by Corollaries \ref{cor-co-1} and \ref{cor-co-2}, composition operators $C_{\varphi_s}$ are all compact from $\calF^p(\C)$ into $\calF^q(\C)$, and $C_{\varphi} = C_{\varphi_1}$ and $C_{\varphi(0)} = C_{\varphi_0}$.
We will show that the map 
$$
[0,1] \to \mathcal C(\mathcal F^p(\C), \mathcal F^q(\C)),\ s \mapsto C_{\varphi_s},
$$
is continuous, that is,
$$
\lim_{s \to s_0}\|C_{\varphi_s} - C_{\varphi_{s_0}}\| = 0, \ \forall s_0 \in [0,1].
$$

\textbf{- Case 1:} $s_0 <1$. In this case fix some $s_1 \in (s_0,1)$. For each $s < s_1$ and each $f \in \calF^p(\C)$ with $\|f\|_p \leq 1$, we have
\begin{align*}
&\|C_{\varphi_s}f - C_{\varphi_{s_0}}f\|_q^q \\
 = & \dfrac{q}{2\pi}\int_{\C} |f(\varphi(sz)) - f(\varphi(s_0z))|^q e^{-\frac{q|z|^2}{2}} dA(z) \\
= & \dfrac{q}{2\pi}\int_{\C}  \left| \int_{s_0}^s z(C_{\varphi}f)'(tz) dt\right|^q e^{-\frac{q|z|^2}{2}} dA(z) \\
\leq & \dfrac{q}{2\pi}\int_{\C} |s-s_0|^q |z|^q \max_{t \in <s_0,s>}|(C_{\varphi}f)'(tz)|^q e^{-\frac{q|z|^2}{2}} dA(z) \\
\leq & e^{2q} \dfrac{q}{2\pi}\int_{\C} |s-s_0|^q|z|^q \max_{t \in <s_0,s>} \left((1+|tz|)e^{\frac{|tz|^2}{2}} \|C_{\varphi}f\|_q\right)^q e^{-\frac{q|z|^2}{2}} dA(z) \\
&\text{( due to Lemma \ref{lem-Fp}[(ii)])}\\
\leq & e^{2q} |s-s_0|^q \|C_{\varphi}f\|_q^q \dfrac{q}{2\pi}\int_{\C} |z|^q  (1+|z|)^q e^{-\frac{q(1-s_1^2)|z|^2}{2}} dA(z) \\
\leq & M^q \|C_{\varphi}\|^q \|f\|_p^q |s-s_0|^q\\ 
\leq & M^q \|C_{\varphi}\|^q |s-s_0|^q,
\end{align*}
where $<s_0,s>$ is the closed interval connecting $s_0$ and $s$ and 
$$
M^q:= e^{2q} \dfrac{q}{2\pi}\int_{\C} |z|^q  (1+|z|)^q e^{-\frac{q(1-s_1^2)|z|^2}{2}} dA(z) < \infty.
$$
From this, it follows that, for every $s < s_1$,
$$
\|C_{\varphi_s} - C_{\varphi_{s_0}}\| = \sup_{\|f\|_p \leq 1} \|C_{\varphi_s}f - C_{\varphi_{s_0}}f\|_q \leq M \|C_{\varphi}\| |s-s_0|,
$$
and the desired limit follows.

\medskip

\textbf{- Case 2:} $s_0 = 1$. Fix an arbitrary number $R>0$. We have that, for every $f \in \calF^p(\C)$ with $\|f\|_p \leq 1$ and every $s \in [1/2,1]$,
\begin{align*}
&\|C_{\varphi_s}f - C_{\varphi_1}f\|_q^q \\
 = & \dfrac{q}{2\pi} \left( \int_{|z|\leq R} + \int_{|z|> R} \right) |f(\varphi(sz)) - f(\varphi(z))|^q e^{-\frac{q|z|^2}{2}} dA(z) \\
= & \mathcal I(R,s) + \mathcal J(R,s).
\end{align*}

\textbf{* Estimate $\mathcal I(R, s)$:} Arguing as above in Case 1, we get
\begin{align*}
\mathcal I(R,s) = & \dfrac{q}{2\pi}\int_{|z|\leq R}  \left| \int_{s}^1 z(C_{\varphi}f)'(tz) dt\right|^q e^{-\frac{q|z|^2}{2}} dA(z) \\
\leq & \dfrac{q}{2\pi}\int_{|z| \leq R} |1-s|^q |z|^q \max_{t \in [s,1]}|(C_{\varphi}f)'(tz)|^q e^{-\frac{q|z|^2}{2}} dA(z) \\
\leq & e^{2q} |1-s|^q \dfrac{q}{2\pi}\int_{|z| \leq R} |z|^q \max_{t \in [s,1]} \left((1+|tz|)e^{\frac{|tz|^2}{2}} \|C_{\varphi}f\|_q\right)^q e^{-\frac{q|z|^2}{2}} dA(z) \\
\leq & e^{2q} |1-s|^q \|C_{\varphi}f\|_q^q \dfrac{q}{2\pi}\int_{|z| \leq R} |z|^q  (1+|z|)^q dA(z) \\
\leq &M_R^q \|C_{\varphi}\|^q |1-s|^q,
\end{align*}
where
$$
M_R^q:= e^{2q} \dfrac{q}{2\pi}\int_{|z| \leq R} |z|^q  (1+|z|)^q dA(z) < \infty.
$$

\textbf{* Estimate $\mathcal J(R, s)$:} For every $f \in \calF^p(\C)$ with $\|f\|_p \leq 1$ and every $s \in [1/2,1]$, using the standard inequality $(x+y)^q \leq 2^q (x^q + y^q)$ for arbitrary positive numbers $x, y, q$, we have
\begin{align*}
\mathcal J(R,s) =& \dfrac{q}{2\pi}\int_{|z| > R} |f(\varphi(sz)) - f(\varphi(z))|^q e^{-\frac{q|z|^2}{2}} dA(z) \\
\leq & 2^q\dfrac{q}{2\pi}\int_{|z| > R} \left( |f(\varphi(sz))|^q + |f(\varphi(z))|^q \right)  e^{-\frac{q|z|^2}{2}} dA(z) \\
\leq & 2^q\dfrac{q}{2\pi}\int_{|z| > R} |f(\varphi(sz))|^q   e^{-\frac{q|sz|^2}{2}} dA(z) +  2^q\dfrac{q}{2\pi}\int_{|z| > R} |f(\varphi(z))|^q   e^{-\frac{q|z|^2}{2}} dA(z) \\
\leq & 2^q\dfrac{q}{2\pi s^2}\int_{|z| > sR} |f(\varphi(z))|^q   e^{-\frac{q|z|^2}{2}} dA(z) +  2^q\dfrac{q}{2\pi}\int_{|z| > R} |f(\varphi(z))|^q   e^{-\frac{q|z|^2}{2}} dA(z) \\
\leq & D_q\dfrac{q}{2\pi}\int_{|z| > R/2} |f(\varphi(z))|^q   e^{-\frac{q|z|^2}{2}} dA(z),
\end{align*}
where $D_q: = 5.2^q$.

We consider the following possibilities.

\textbf{-} If $0 < p \leq q < \infty$ then for every function $f \in \calF^p(\C)$ with $\|f\|_p \leq 1$, by Lemma \ref{lem-pq}, 
\begin{align*}
\mathcal J(R,s) \leq & D_q \dfrac{q}{2\pi} \sup_{|z|>R/2} m_{z}(1,\varphi)^q \int_{|z| > R/2} |f(\varphi(z))|^q e^{-\frac{q|\varphi(z)|^2}{2}} dA(z) \\ 
\leq & D_q \dfrac{\|f\|_q^q}{|a|^2} \sup_{|z|>R/2} m_{z}(1,\varphi)^q \\
\leq & D_q \dfrac{q}{p|a|^2} \|f\|_p^q \sup_{|z|>R/2} m_{z}(1,\varphi)^q \\
\leq & D_q \dfrac{q}{p|a|^2} \sup_{|z|>R/2} m_{z}(1,\varphi)^q.
\end{align*}

Consequently, for every $s \in [1/2,1]$,
\begin{align*}
\|C_{\varphi_s} - C_{\varphi_1}\|^q \leq & \sup_{\|f\|_p \leq 1} (\mathcal I(R,s) + \mathcal J(R,s)) \\
\leq & M_R^q \|C_{\varphi}\|^q |1-s|^q + D_q \dfrac{q}{p|a|^2} \sup_{|z|>R/2} m_{z}(1,\varphi)^q.
\end{align*}
Letting $s \to 1^-$, we obtain
$$
\limsup_{s \to 1^-} \|C_{\varphi_s} - C_{\varphi_1}\|^q  \leq D_q \dfrac{q}{p|a|^2} \sup_{|z|>R/2} m_{z}(1,\varphi)^q.
$$
Since $|a|<1$, $\lim_{|z| \to \infty} m_z(1, \varphi) =0$. Then, letting $R \to \infty$, we get
$$
\limsup_{s \to 1^-} \|C_{\varphi_s} - C_{\varphi_1}\|^q  \leq D_q \dfrac{q}{p|a|^2} \lim_{R \to \infty} \sup_{|z|>R/2} m_{z}(1,\varphi)^q = 0.
$$

\textbf{-} If $0 < q < p < \infty$ then arguing as in Theorem \ref{thm-bd2} and using H\"older inequality, we have that, for every $f \in \calF^p(\C)$ with $\|f\|_p \leq 1$, 
\begin{align*}
\mathcal J(R,s) \leq & D_q \dfrac{q}{2\pi}\int_{|z| > R/2} |f(\varphi(z))|^q   e^{-\frac{q|z|^2}{2}} dA(z)\\
 = & D_q \dfrac{q}{2\pi} \int_{|z| > R/2} m_z(1, \varphi)^q |f(\varphi(z))|^q e^{-\frac{q|\varphi(z)|^2}{2}} dA(z) \\
\leq & D_q \dfrac{q}{2\pi} \left( \int_{|z| > R/2}|f(\varphi(z))|^p e^{-\frac{p|\varphi(z)|^2}{2}} dA(z) \right)^{\frac{q}{p}} \left( \int_{|z| > R/2} m^{\frac{pq}{p-q}}_z(1,\varphi) dA(z) \right)^{\frac{p-q}{p}} \\
\leq & D_q \dfrac{q}{2\pi} \left( \dfrac{2\pi}{p |a|^2} \right)^{\frac{q}{p}} \|f\|_p^q \left( \int_{|z| > R/2} m^{\frac{pq}{p-q}}_z(1,\varphi) dA(z) \right)^{\frac{p-q}{p}} \\
\leq & D_q \dfrac{q}{2\pi} \left( \dfrac{2\pi}{p |a|^2} \right)^{\frac{q}{p}} \left( \int_{|z| > R/2} m^{\frac{pq}{p-q}}_z(1,\varphi) dA(z) \right)^{\frac{p-q}{p}}.
\end{align*}

Consequently, for every $s \in [1/2,1]$,
\begin{align*}
\|C_{\varphi_s} - C_{\varphi_1}\|^q \leq & \sup_{\|f\|_p \leq 1} (\mathcal I(R,s) + \mathcal J(R,s)) \\
\leq & M_R^q \|C_{\varphi}\|^q |1-s|^q + D_q \dfrac{q}{2\pi} \left( \dfrac{2\pi}{p |a|^2} \right)^{\frac{q}{p}} \left( \int_{|z| > R/2} m^{\frac{pq}{p-q}}_z(1,\varphi) dA(z) \right)^{\frac{p-q}{p}}.
\end{align*}
Letting $s \to 1^-$ in the last inequality, we obtain
$$
\limsup_{s \to 1^-} \|C_{\varphi_s} - C_{\varphi_1}\|^q  \leq D_q \dfrac{q}{2\pi} \left( \dfrac{2\pi}{p |a|^2} \right)^{\frac{q}{p}} \left( \int_{|z| > R/2} m^{\frac{pq}{p-q}}_z(1,\varphi) dA(z) \right)^{\frac{p-q}{p}}.
$$
Since $C_{\varphi}$ is compact from $\mathcal F^p(\C)$ into $\mathcal F^q(\C)$, by Theorem \ref{thm-bd2}, $m_z(1, \varphi) \in L^{\frac{pq}{p-q}}(\C, dA)$. Then, letting $R \to \infty$, we get
$$
\limsup_{s \to 1^-} \|C_{\varphi_s} - C_{\varphi_1}\|^q  \leq D_q \dfrac{q}{2\pi} \left( \dfrac{2\pi}{p |a|^2} \right)^{\frac{q}{p}} \lim_{R \to \infty} \left( \int_{|z| > R/2} m^{\frac{pq}{p-q}}_z(1,\varphi) dA(z) \right)^{\frac{p-q}{p}} =0.
$$

The proof is completed.
\end{proof}

\begin{thm} \label{thm-path-co}
For every $p, q \in (0, \infty)$, the set of all compact composition operators from $\calF^p(\C)$ into $\mathcal F^q(\C)$ is a path connected component of $\calC \big(\calF^p(\C),\calF^q(\C)\big)$.
\end{thm}
\begin{proof}
Let $C_{\varphi_1}$ and $C_{\varphi_2}$ be two compact composition operators from $\calF^p(\C)$ into $\mathcal F^q(\C)$.
By Proposition \ref{prop-path-c}, $C_{\varphi_1}$ and $C_{\varphi_1(0)}$ belong to the same path connected component and so do $C_{\varphi_2}$ and $C_{\varphi_2(0)}$. 
We show that $C_{\varphi_1(0)}$ and $C_{\varphi_2(0)}$ are in the same path connected component.

For each $s \in [0,1]$, put
$$
\beta_s: = (1-s)\varphi_1(0) + s\varphi_2(0),
$$
and
$$
B_1: = \varphi_2(0) - \varphi_1(0), B_2: = |\varphi_1(0)| + |\varphi_2(0)|.
$$
Obviously, $|\beta_s| \leq |B_2|$ for every $s \in [0,1]$. Moreover, $C_{\varphi_1(0)} = C_{\beta_0}$ and $C_{\varphi_2(0)} = C_{\beta_1}$ and the composition operators $C_{\beta_s}$ are compact from $\calF^p(\C)$ into $\mathcal F^q(\C)$ for all $s \in [0,1]$.

We now prove that the map 
$$
[0,1] \to \mathcal C(\mathcal F^p(\C), \mathcal F^q(\C)),\ s \mapsto C_{\beta_s}
$$
is continuous. 
Fix $s_0 \in [0,1]$. For every $s\in [0,1]$ and every $f \in \calF^p(\C)$ with $\|f\|_p \leq 1$, using Lemma \ref{lem-Fp}[(ii)],  we have
\begin{align*}
\|C_{\beta_s} f - C_{\beta_{s_0}} f\|_q^q & =\dfrac{q}{2 \pi} |f(\beta_s) - f(\beta_{s_0})|^q \int_{\C} e^{-\frac{q|z|^2}{2}}dA(z)\\
& = \|1\|_q^q \left| \int_{s_0}^s \left(B_1 f'(\beta_t) \right) dt \right|^q \\
& \leq  \|1\|_q^q |s-s_0|^q \max_{t \in <s_0,s>} \left| B_1 f'(\beta_t) \right|^q \\
& \leq  \|1\|_q^q |s-s_0|^q |B_1|^q \max_{t \in <s_0,s>} |f'(\beta_t)|^q \\
& \leq  \|1\|_q^q |s-s_0|^q |B_1|^q \left( \|f\|_p e^{2} (1+ B_2)  e^{\frac{B_2^2}{2}} \right)^q\\
& \leq  M^q |s-s_0|^q,
\end{align*}
where, as above, $<s_0, s>$ is the closed interval connecting $s_0$ and $s$, and
$$
M: = \|1\|_q e^2|B_1|(1+ B_2) e^{\frac{B_2^2}{2}}.
$$
From this, it follows that, for every $s \in [0,1]$,
$$
\|C_{\beta_s} - C_{\beta_{s_0}} \| = \sup_{\|f\|_p\leq 1} \|C_{\beta_s} f - C_{\beta_{s_0}} f\|_q \leq M |s-s_0|.
$$
It implies that
$$
\lim_{s \to s_0}\|C_{\beta_s} - C_{\beta_{s_0}} \| =0.
$$

Consequently, $C_{\varphi_1}$ and $C_{\varphi_2}$ are in the same path connected component of $C(\mathcal F^p(\C), \mathcal F^q(\C))$, which completes the proof.
\end{proof}

From Theorem \ref{thm-path-co} and Corollary \ref{cor-co-2} we get

\begin{cor}\label{cor-path-co}
If $0 < q < p < \infty$, then the space $C(\mathcal F^p(\C), \mathcal F^q(\C))$ is path connected.
\end{cor}

Next for $0 < p \leq q < \infty$ we give the following result about the characterization of isolated composition operators $C_{\varphi}$ in the space $\mathcal C(\mathcal F^p(\C), \mathcal F^q(\C))$. 
\begin{thm}\label{thm-ip-co}
Let $0 < p \leq q < \infty$ and $C_{\varphi}$ be a bounded composition operator from $\mathcal F^p(\C)$ into $\mathcal F^q(\C)$ induced by $\varphi(z) = az+b$ with $|a| \leq 1$. The following conditions are equivalent:
\begin{itemize}
\item[(i)] $C_{\varphi}$ is isolated in $\mathcal C\big(\mathcal F^p(\C), \mathcal F^q(\C)\big)$;
\item[(ii)] $C_{\varphi}$ is non-compact, that is, $|a| = 1$ and $b=0$;
\item[(iii)] $\|C_{\varphi} - C_{\phi}\| \geq 1$ for all affine functions $\phi(z) = cz + d \neq \varphi(z)$ such that $C_{\phi} \in \mathcal C\big(\mathcal F^p(\C), \mathcal F^q(\C)\big)$.
\end{itemize}
\end{thm}
\begin{proof} 

(i) $\Longrightarrow$ (ii). By Theorem \ref{thm-path-co}, if $C_{\varphi}$ is an isolated composition operator in $\mathcal C(\mathcal F^p(\C),\mathcal F^q(\C))$, then $C_{\varphi}$ must be non-compact. And hence, by Corollary \ref{cor-co-1}, $|a|=1$ and $b=0$. 

(ii) $\Longrightarrow$ (iii). Assume that $|a|=1$ and $b=0$. In this case, for every affine function $\phi(z) = c z+ d \neq \varphi(z)$ such that $C_{\phi} \in \mathcal C\big(\mathcal F^p(\C), \mathcal F^q(\C)\big)$, by Lemma \ref{lem-Fp}[(i)], we have that, for all $w, z\in\C$,
\begin{eqnarray*}
\|C_{\varphi}k_{w} - C_{\phi}k_{w}\|_q%
&\ge& \left|C_{\varphi}k_{w}(z)- C_{\phi}k_{w}(z)\right| e^{-\frac{|z|^2}{2}}\\
&=& \left|e^{\overline{w}\varphi(z)-\frac{|w|^2+|z|^2}{2}} - e^{\overline{w}\phi(z)-\frac{|w|^2+|z|^2}{2}}\right|.
\end{eqnarray*}
In particular, with $w=\varphi(z)$, the last inequality gives
\begin{align*}
\|C_{\varphi}k_{\varphi(z)} - C_{\phi}k_{\varphi(z)}\|_q &\ge \left| 1 - \big| e^{\overline{az}(c z + d) - |z|^2}\big|\right| \notag \\
& = \left| 1 - e^{\re(\overline{a}c - 1)|z|^2 + \re(d\overline{az})} \right|, \forall z \in \mathbb C. \notag
\end{align*}
Since $C_{\phi} \in \mathcal C\big(\mathcal F^p(\C), \mathcal F^q(\C)\big)$ and $\varphi(z) \neq \phi(z)$, then $|c| \leq 1$ and $c \neq a$. Hence, $\re(\overline{a}c) <1$. 
From this it follows that
\begin{align*}
\|C_{\varphi} - C_{\phi}\| &\geq \limsup_{|z| \to \infty}\|C_{\varphi}k_{\varphi(z)} - C_{\phi}k_{\varphi(z)}\|_q \\
&\ge \limsup_{|z| \to \infty} \left| 1 - e^{\re(\overline{a}c - 1)|z|^2 + \re(d\overline{az})} \right| =1.
\end{align*}

(iii) $\Longrightarrow$ (i) is obvious.
\end{proof}

From Theorems \ref{thm-path-co} and \ref{thm-ip-co}, we immediately get the following result.

\begin{cor}
Let $0 < p \leq q < \infty$. The connected component and path connected component in $\mathcal C\big(\mathcal F^p(\C), \mathcal F^q(\C)\big)$are the same and they are only the set of all compact composition operators from $\mathcal F^p(\C)$ into $\mathcal F^q(\C)$.
\end{cor}

\subsection{The space $\calC_w\big(\calF^p(\C),\calF^q(\C)\big)$.}
In this subsection using the results in Subsection \ref{topo-c} we obtain a complete characterization of the component structure of $\calC_w\big(\calF^p(\C),\calF^q(\C)\big)$.

For $p, q \in (0, \infty)$ and $\varphi(z) = az + b$ with $|a| \leq 1$ we denote by $\calF(\varphi, p, q)$ the set of all nonzero functions $\psi \in \calF^q(\C)$ such that $W_{\psi, \varphi}: \calF^p(\C) \to \calF^q(\C)$ is bounded. Then
$$
\calC_w\big(\calF^p(\C),\calF^q(\C)\big) = \{ W_{\psi, \varphi}: \psi \in \calF(\varphi, p, q), \varphi(z) = az + b, |a| \leq 1 \}.
$$

\begin{lem}\label{lem-wco}
Let $p, q \in (0,\infty)$, $\varphi(z) = az + b$ with $|a| \leq 1$ and $\psi_1, \psi_2 \in \calF(\varphi, p, q)$. Then $W_{\psi_1,\varphi}$ and $W_{\psi_2,\varphi}$ are in the same path connected component of $\calC_w\big(\calF^p(\C),\calF^q(\C)\big)$.
\end{lem}

\begin{proof}
We can easily show that there is a complex valued continuous function $\alpha(t)$ on $[0,1]$ such that $\alpha(0)=0, \alpha(1) = 1$ and $u_t: = (1-\alpha(t))\psi_1 + \alpha(t)\psi_2$ are all nonzero functions in $\calF(\varphi, p, q)$ for all $t \in [0,1]$. 

Indeed, if $\psi_2(z) = \lambda \psi_1(z)$ for some $\lambda \in \C\setminus\{0,1\}$ and all $z \in \C$, we can take any continuous function $\alpha(t)$ so that $\alpha(t) \neq 1/(1-\lambda)$ for all $t \in [0,1]$. Otherwise, we put $\alpha(t) =t$. Moreover, for each $t \in [0,1]$,
$W_{u_t, \varphi} =  (1 - \alpha(t)) W_{\psi_1, \varphi} +  \alpha(t) W_{\psi_2, \varphi}$, and hence, $u_t \in \calF(\varphi, p, q)$.

Obviously,
$$
W_{\psi_1, \varphi} = W_{u_0, \varphi}, W_{\psi_2, \varphi} = W_{u_1, \varphi},
$$
and, for every $s,t \in [0,1]$ and $f \in \calF^p(\C)$,
$$
W_{u_s, \varphi}f - W_{u_t, \varphi}f = (u_s - u_t) f\circ \varphi = (\alpha(s) - \alpha(t)) W_{\psi_2-\psi_1,\varphi}f.
$$
From this, it follows that, for every $t \in [0,1]$,
$$
\lim_{s\to t} \|W_{u_s, \varphi} - W_{u_t, \varphi}\| = 0.
$$
This means that the map
$$
[0,1] \to \mathcal C_w(\mathcal F^p(\C), \mathcal F^q(\C)),\ t \mapsto W_{u_t, \varphi},
$$
is continuous. The proof is completed.
\end{proof}

Let
$$
\mathcal S_0: = \{ \varphi(z) = az + b: |a| < 1\} \text{ and } \mathcal S_1: = \{ \varphi(z) = az + b: |a| = 1\},
$$
and
$$
\mathcal C_{w,0}(\calF^p(\C), \calF^q(\C)): = \{W_{\psi, \varphi}: \varphi \in \mathcal S_0, \psi \in \calF(\varphi, p, q) \}.
$$
\begin{thm}\label{thm-path-wc}
Let $p, q \in (0, \infty)$ be given.
\begin{itemize}
\item[(a)] If $ 0 < q < p < \infty$, then the space $\calC_w\big(\calF^p(\C),\calF^q(\C)\big)$ is path connected.
\item[(b)] If $ 0 < p \leq q < \infty$, then the space $\calC_w\big(\calF^p(\C),\calF^q(\C)\big)$ has the following path connected components
$$
\calC_w\big(\calF^p(\C),\calF^q(\C)\big) = \mathcal C_{w,0}(\calF^p(\C), \calF^q(\C)) \bigcup \bigcup_{\varphi \in \mathcal S_1} \{W_{\psi, \varphi}: \psi \in \calF(\varphi, p, q)\}. 
$$
\end{itemize}
\end{thm}

\begin{proof}
(a) First we note that for every $W_{\psi,\varphi} \in \calC_w(\calF^p(\C), \calF^q(\C))$,  by Proposition \ref{prop-nec} and Theorem \ref{thm-bd2}, 
$$
m(\psi, \varphi) < \infty \text{ and } m_z(\psi, \varphi) \in L^{\frac{pq}{p-q}}(\C,dA).
$$
It implies that $\varphi(z) = az + b$ with $|a| < 1$. Indeed, if $|a| = 1$ then $\psi(z) = \psi(0) e^{- \overline{b}az}$ and $m_z(\psi, \varphi) = |\psi(0)|e^{\frac{|b|^2}{2}}$ for all $z \in \C$ which is impossible.   

Then, by Corollary \ref{cor-co-2}, $W_{1, \varphi} = C_{\varphi}$ also belongs to $\calC_w(\calF^p(\C), \calF^q(\C))$. Hence, by Lemma \ref{lem-wco}, $W_{\psi, \varphi}$ and $C_{\varphi}$ are in the same path connected component of $\calC_w\big(\calF^p(\C),\calF^q(\C)\big)$. 

From this and Corollary \ref{cor-path-co} it follows that $\calC_w\big(\calF^p(\C),\calF^q(\C)\big)$ is path connected. 

(b) Fix an arbitrary pair of weighted composition operators $W_{\psi_1, \varphi_1}$ and $W_{\psi_2, \varphi_2}$ in $\calC_{w,0}\big(\calF^p(\C),\calF^q(\C)\big)$. Then, by Corollary \ref{cor-co-1}, $W_{1, \varphi_1} = C_{\varphi_1}$ and $W_{1, \varphi_2} = C_{\varphi_2}$ are compact from $\calF^p(\C)$ into $\calF^q(\C)$. And hence, by Lemma \ref{lem-wco}, $W_{\psi_1, \varphi_1}$ and $C_{\varphi_1}$ belong to the same path connected component of $\calC_w\big(\calF^p(\C),\calF^q(\C)\big)$, and so do $W_{\psi_2, \varphi_2}$ and $C_{\varphi_2}$.

From this and Theorem \ref{thm-path-co} it follows that $W_{\psi_1, \varphi_1}$ and $W_{\psi_2, \varphi_2}$ belong to the same path connected component of $\calC_w\big(\calF^p(\C),\calF^q(\C)\big)$. This means that the set $\calC_{w,0}\big(\calF^p(\C),\calF^q(\C)\big)$ is a path connected component in $\calC_w\big(\calF^p(\C),\calF^q(\C)\big)$.

Next, for each $\varphi(z) = a z + b \in \mathcal S_1$, $\psi \in \calF(\varphi,p,q)$ if and only of $\psi(z) = \psi(0) e^{- \overline{b}az}$.
Then
$$
\{W_{\psi, \varphi}: \psi \in \calF(\varphi, p, q)\} = \{W_{c \psi_0, \varphi}: c \in \C \} \text{ with } \psi_0(z) := e^{- \overline{b}az}.
$$
Thus, for each $\varphi \in \mathcal S_1$, the set  $\{W_{\psi, \varphi}: \psi \in \calF(\varphi, p, q)\}$ is a path connected component in $\calC_{w}\big(\calF^p(\C),\calF^q(\C)\big)$.
\end{proof}

Finally, it should be noted that in the space $\mathcal \mathcal C_w(\mathcal F^p(\C), \mathcal F^q(\C))$ there does not exist an isolated weighted composition operator. Indeed, for every weighted composition operator $W_{\psi, \varphi}$,
$$
\lim_{s \to 1} \|W_{s\psi, \varphi} - W_{\psi, \varphi}\| = \lim_{s \to 1}|s-1| \|W_{\psi, \varphi}\| =0.
$$
%%%%%%%%%%%%%%%%%%%%%%%%%%%%%%%%%%%%%%%%%%%%%%%%%%%%%%%%%%%%%%%%%%

\end{document}